\pdfoutput=1
\newif\ifpreprint
\preprinttrue 
\ifdefined\ispreprint
  \preprinttrue
\fi
\ifdefined\isjournal
  \preprintfalse
\fi
\ifpreprint

\documentclass[a4paper,10pt]{article}
\usepackage[T1]{fontenc}
\usepackage{lmodern}
\usepackage[a4paper]{geometry}

\else

  \documentclass{svmult} 

\fi

\usepackage{shellesc}
\usepackage[utf8]{inputenc}
\usepackage[T1]{fontenc}
\usepackage{microtype}
\usepackage[hyphens]{url}
\usepackage{amsmath,amsthm}
\usepackage{amsfonts,amssymb}
\usepackage{mathdots}

\usepackage[normalem]{ulem} 

\usepackage[hidelinks,hypertexnames=false,hyperindex=true,pdfpagelabels,linktoc=all]{hyperref}
\usepackage{graphicx}
\usepackage{tikz,pgfplots}
\usepackage{pgfplotstable}
\usepgfplotslibrary{groupplots}
\usetikzlibrary{matrix}
\pgfplotsset{compat=newest} %
\usepgfplotslibrary{external} 
\usepackage{pgffor}
\usepackage[numbers,sort]{natbib}

\pgfplotsset{table/search path={data},}

\newcommand{\imagunit}{\mathrm{i}}
\newcommand{\bszero}{\boldsymbol{0}} 
\newcommand{\bsh}{{\boldsymbol{h}}}    
\newcommand{\bsg}{\boldsymbol{g}} 
\newcommand{\bsp}{\boldsymbol{p}}    
 
\newcommand{\bsx}{\boldsymbol{x}}    
\newcommand{\bsu}{\boldsymbol{u}}
\newcommand{\bsy}{\boldsymbol{y}}    
\newcommand{\bsz}{\boldsymbol{z}}    

\newcommand{\Z}{\mathbb{Z}}

\newcommand{\R}{\mathbb{R}}

\newcommand{\C}{\mathbb{C}}
\newcommand{\calA}{\mathcal{A}}

\newcommand{\tpmod}[1]{~(\operatorname{mod}{#1})} 
 
\newcommand{\twopii}{2\pi \imagunit\,}
\newcommand*{\T}{\mathbb{T}}

\newcommand{\rme}{\mathrm{e}}
\newcommand{\ad}{\mathrm{ad}}
\newcommand{\tsc}{\gamma} 

\DeclareMathOperator{\diag}{diag}

\newcommand{\Deltat}{{\Delta t}}

\definecolor{darkred}{RGB}{139,0,0}
\definecolor{darkgreen}{RGB}{0,130,70}
\definecolor{darkmagenta}{RGB}{139,0,139}
\definecolor{darkorange}{RGB}{180,60,0}

\newtheorem{theorem}{Theorem}
\newtheorem{lemma}[theorem]{Lemma}

\newcommand{\logLogSlopeTriangle}[5]
{

    \pgfplotsextra
    {
        \pgfkeysgetvalue{/pgfplots/xmin}{\xmin}
        \pgfkeysgetvalue{/pgfplots/xmax}{\xmax}
        \pgfkeysgetvalue{/pgfplots/ymin}{\ymin}
        \pgfkeysgetvalue{/pgfplots/ymax}{\ymax}

        \pgfmathsetmacro{\xArel}{#1}
        \pgfmathsetmacro{\yArel}{#3}
        \pgfmathsetmacro{\xBrel}{#1-#2}
        \pgfmathsetmacro{\yBrel}{\yArel}
        \pgfmathsetmacro{\xCrel}{\xArel}

        \pgfmathsetmacro{\lnxB}{\xmin*(1-(#1-#2))+\xmax*(#1-#2)} 
        \pgfmathsetmacro{\lnxA}{\xmin*(1-#1)+\xmax*#1} 
        \pgfmathsetmacro{\lnyA}{\ymin*(1-#3)+\ymax*#3} 
        \pgfmathsetmacro{\lnyC}{\lnyA+#4*(\lnxA-\lnxB)}
        \pgfmathsetmacro{\yCrel}{\lnyC-\ymin)/(\ymax-\ymin)} 

        \coordinate (A) at (rel axis cs:\xArel,\yArel);
        \coordinate (B) at (rel axis cs:\xBrel,\yBrel);
        \coordinate (C) at (rel axis cs:\xCrel,\yCrel);

        \draw[#5]   (A)-- node[pos=0.5,anchor=north] {1}
                    (B)-- 
                    (C)-- node[pos=0.5,anchor=west] {#4}
                    cycle;
    }
}
  \newcommand{\email}[1]{\protect\href{mailto:#1}{#1}}
\allowdisplaybreaks
\pgfplotsset{compat=1.13}

\title{Rank-$1$ lattices and higher-order exponential splitting for the time-dependent Schrödinger equation}
\author{Yuya Suzuki\thanks{NUMA section, Computer Science department, KU Leuven, Belgium (\email{yuya.suzuki@cs.kuleuven.be}, \email{dirk.nuyens@cs.kuleuven.be}).} \and Dirk Nuyens\footnotemark[1]}

\begin{document}


\maketitle

\abstract{
In this paper, we propose a numerical method to approximate the solution of the time-dependent Schrödinger equation with periodic boundary condition in a high-dimensional setting. We discretize space by using the Fourier pseudo-spectral method on rank-$1$ lattice points, and then discretize time by using a higher-order exponential operator splitting method.
In this scheme the convergence rate of the time discretization  depends on properties of the spatial discretization. We prove that the proposed method, using rank-$1$ lattice points in space, allows to obtain higher-order time convergence, and, additionally, that the necessary condition on the space discretization can be independent of the problem dimension~$d$.
We illustrate our method by numerical results from 2 to 8 dimensions which show that such higher-order convergence can really be obtained in practice.
} 


\section{Introduction}

Rank-$1$ lattice points have been widely used in the context of high-dimensional problems. Their traditional usage is in numerical integration, see, e.g., \cite{MR2479214,MR1442955} and references therein. In this work, we use rank-$1$ lattice points for function approximation, to approximate the solution of the time-dependent Schrödinger equation (TDSE). Function approximation using rank-$1$ lattice points has recently received more attention, see, e.g., \cite{MR3671595,MR2400322,KWW09,MR2011715,2018arXiv180806357S,MR2878018}.
In \cite{MR2011715}, Li and Hickernell introduced the pseudo-spectral Fourier collocation method using rank-$1$ lattice rules. Due to the rank-$1$ lattice structure, Fourier pseudo-spectral methods can be efficiently implemented using one-dimensional Fast Fourier transformations (FFTs). This is well known, and we state the exact form in Theorem~\ref{thm:rank-1-properties} together with other useful properties of approximations on rank-$1$ lattice points.  

To simulate many particles in the quantum world is a computationally challenging problem. For the TDSE, the dimensionality of the problem increases with the number of particles of the system. In the present paper, the following form is considered:
\begin{align}
  \imagunit \, \tsc \, \frac{\partial u}{\partial t}(\bsx, t)
  &=
  -\frac{ \tsc^2}{2} \, \Delta u(\bsx, t) + v{(\bsx)} \, u(\bsx, t),
  \label{eq:tdse} 
  \\
  \notag
  u(\bsx,0)
  &=
  g(\bsx)
  ,
\end{align}
where $\imagunit$ represents the imaginary unit, $\bsx$ is the spatial position in the $d$-dimensional torus $\T^d = \T([0,1]^d) \simeq [0,1]^d$, the time $t$ is positive valued, and $\tsc$ is a small positive parameter. The function $u(\bsx,t)$ is the sought solution, while $v(\bsx)$ and $g(\bsx)$ are the potential and initial conditions respectively. 
The Laplacian can be interpreted as $\Delta = \sum_{i=1}^M \sum_{j=1}^D \partial^2 / \partial x_{i,j}^2=\sum_{i=1}^d \partial^2 / \partial x_{i}^2$ where $M$ is the number of particles and $D$ is the physical dimensionality with $MD=d$.
We note that the above form of the TDSE becomes equivalent after substitution to the following form which is common in the context of physics:
\begin{align*}
  \imagunit \, \hbar \, \frac{\partial \psi}{\partial t}(\bsx,t)
  &=
  -\frac{\hbar^2}{2m} \,  \Delta \psi(\bsx, t) + v(\bsx) \, \psi(\bsx, t),
\end{align*}%
where $\hbar$ is the reduced Planck constant and $m$ is the mass.

The form~\eqref{eq:tdse} of the TDSE has been studied from various perspectives of numerical analysis \cite{MR1799313,MR2377257,MR2399406,MR2486135}. In the present paper, we focus on two perspectives; high-dimensionality and higher-order convergence in time stepping.
For the first point, Gradinaru \cite{MR2377257} proposed to use sparse grids for the physical space. In \cite{2018arXiv180806357S}, the current authors used rank-$1$ lattice points to prove second order convergence for the time discretization using Strang splitting and numerically compared results with the sparse grid approach from \cite{MR2377257}. The numerical result using rank-$1$ lattices showed the expected second order convergence even up to $12$~dimensions. Hence rank-$1$ lattice points perform thereby much better than the sparse grid approach.
The second point, higher-order convergence in time stepping, is successfully achieved by Thalhammer \cite{MR2399406} using higher-order exponential operator splitting. In that paper, the spatial discretization was done by a full grid and therefore was limited to lower dimensional cases ($d\leq3$).

The rest of this paper is organized as follows: Section~\ref{sec:numerical-method} describes the proposed method consisting of the higher-order exponential splitting method and Fourier pseudo-spectral method using rank-$1$ lattices. 
Section~\ref{sec:numerical} shows numerical results with various settings. The main aim here is to show higher-order time stepping convergence in higher-dimensional cases. 
Finally, Section~\ref{sec:conclusion} concludes the present paper with a short summary.

Throughout the present paper, we denote the set of integer numbers by $\Z$ and the ring of integers modulo $n$ by $\Z_n:= \{0,1,...,n-1\}$.
We distinguish between the normal equivalence in congruence modulo $n$ as $a \equiv b\pmod{n}$ and the binary operation modulo $n$ denoted by $\bmod{n}$ which returns the corresponding value in $\Z_n$ for $\bmod{n}$ and in $\T$ for $\bmod{1}$.

\section{The numerical method}\label{sec:numerical-method}

In this section, we describe necessary ingredients of our method. For the conciseness, we restrict ourselves to the rank-$1$ lattice points instead of general rank-$r$ lattice points. However, our method is indeed possible to generalize to rank-$r$ lattice points, similar as in \cite{2018arXiv180806357S}.

We use a rank-$1$ lattice point set and an associated anti-aliasing set for the Fourier pseudo-spectral method. For using the Fourier pseudo-spectral method, one obvious choice is regular grids \cite{MR2399406}, but the number of points increases too quickly in terms of the number of dimensions. To mitigate this problem, Gradinaru \cite{MR2377257,MR2308834} proposed to use sparse grids.
For the same reason we introduced lattice points in \cite{2018arXiv180806357S} to get first and second order time convergence, and obtained much better results compared to \cite{MR2377257,MR2308834}.

\subsection{Rank-$1$ lattice point sets and the associated anti-aliasing sets}

 A rank-$1$ lattice point set $\Lambda(\bsz, n)$ is fully determined by the modulus $n$ and a generating vector $\bsz \in \Z_n^d$:
\[
  \Lambda(\bsz, n)
  :=
  \left\{ \frac{\bsz k}{n} \bmod 1 \mathrel{:} k \in \Z \right\}
  .
\]
Usually, all components of the generating vector are chosen to be relatively prime to $n$ which means all points have different values in each coordinate and the number of points is exactly $n$.
The generating vector determines the \emph{quality} of the rank-$1$ lattice points. 
Of course, the quality criterion needs to take into account what the lattice points will be used for. A well studied setting is numerical integration, e.g., \cite{MR1442955,Nuy2014} and \cite[Chapter~5]{MR1172997}. 
Function approximation using lattice points is relatively new. In that context, we refer to \cite{MR2400322,KWW09,MR3671595}. 
We call $\calA (\bsz,n) \subset \Z^d$ an anti-aliasing set for the lattice point set $\Lambda(\bsz,n)$ if
\[
  \bsz \cdot \bsh
  \not\equiv
  \bsz \cdot \bsh'
  \pmod{n}
  \quad
  \text{for all } \bsh, \bsh' \in \calA(\bsz, n), \; \; \bsh \ne \bsh'
  .
\]%
We remark that the anti-aliasing set is not uniquely determined and the cardinality $| \calA (\bsz,n) | \leq n$. By using the \emph{dual lattice} $\Lambda^{\bot}(\bsz,n):=\{\bsh\in \Z^d : \bsz\cdot\bsh \equiv 0 \pmod{n}\}$, we can rewrite the condition as $\bsh-\bsh'\not\in \Lambda^{\bot}(\bsz,n)$ for $\bsh \ne \bsh'$. 
If we have the full cardinality $| \calA (\bsz,n) | = n$, we can divide $\Z^d$ into conjugacy classes:
\newlength\sumd
\settowidth{\sumd}{$\scriptstyle \bsh \in \Lambda^\perp(\bsz,n)$}
\begin{align}
\begin{split}
  \Z^d
  &=
  \biguplus_{\makebox[\sumd]{$\scriptstyle \bsh \in \Lambda^\perp(\bsz,n)$}} \left( \bsh + \calA(\bsz,n) \right)
  \\
  &=
  \biguplus_{\makebox[\sumd]{$\scriptstyle \bsh \in \calA(\bsz,n)$}} \{ \bsh' \in \Z^d : \bsz \cdot \bsh' \equiv \bsz \cdot \bsh \pmod{n} \}
  \\
  &=
  \biguplus_{\makebox[\sumd]{$\scriptstyle j\in\Z_{n}$}} \{ \bsh \in \Z^d : \bsz \cdot \bsh \equiv j \pmod{n} \}
  ,
  \end{split}
  \label{eq:disjunct}
\end{align}
where $\uplus$ is the union of conjugacy classes.

\subsection{Korobov spaces}

Rank-$1$ lattices are closely related to \emph{Korobov spaces} which are reproducing kernel Hilbert spaces of Fourier series.
The Korobov space $E_\alpha(\T^d)$ is given by
 \[
  E_\alpha(\T^d)
  := 
  \left\{ f \in L_2(\T^d) \mathrel{:} 
    \|f\|^2_{E_\alpha(\T^d)}
    :=
    \sum_{\bsh\in \Z^d} |\widehat{f}(\bsh)|^2 \, r^2_{\alpha}(\bsh) 
    <
    \infty
  \right\},
\]
where
\begin{align}\label{eq:ralpha}
  r^2_{\alpha}(\bsh)
  :=
  \prod_{j=1}^d \max(|h_j|^{2\alpha},1)
  .
\end{align}
The parameter $\alpha \ge 1/2$ is called the smoothness parameter which determines the rate of decay of the Fourier coefficients.
To ensure regularity of the solution of the TDSE~\eqref{eq:tdse} and to prove that our method gives higher-order convergence for the temporal discretization, we will assume that the initial condition $g(\bsx)$ and the potential function $v(\bsx)$ are in the Korobov space with given smoothness, see Lemma~\ref{lem:u-series} and Theorem~\ref{thm:pth-commu}.

\subsection{Fourier pseudo-spectral methods using rank-$1$ lattices}

We approximate the solution of the TDSE~\eqref{eq:tdse} by the truncated Fourier series. To ensure the solution to be regular enough so that the Fourier expansion makes sense (e.g., uniqueness, continuity, point-wise convergence), 
we require all functions to be in \emph{Wiener algebra} $A(\T^d)$:
\begin{equation*}
  A(\T^d)
  :=
  \{f\in L_2(\T^d) : \;
  \|f\|_{A(\T^d)}
  :=
  \sum_{\bsh\in \Z^d} |\widehat{f}(\bsh)| < \infty
  \}
  .
\end{equation*}
For $\alpha > 1/2 $, we have $E_\alpha(\T^d) \subset A(\T^d)$. 
The following lemma shows the regularity of the solution, and the TDSE~\eqref{eq:tdse} in terms of Fourier coefficients and was already stated and proven in \cite{2018arXiv180806357S}.
\begin{lemma}[Regularity of solution and Fourier expansion] \label{lem:u-series}
Given the TDSE~\eqref{eq:tdse} with $v,g \in E_\alpha(\T^d)$ and $\alpha \ge 2$, then the solution $u(\bsx,t) \in E_\alpha(\T^d)$ for all finite $t \ge 0$ and therefore
\begin{align}\label{eq:u-series}
  u(\bsx,t)
  =
  \sum_{\bsh \in \Z^d} \widehat{u}(\bsh,t) \, \exp(\twopii \bsh \cdot \bsx)
  ,
\end{align}
with
\begin{equation}\label{eq:Fourier-ODE}
  \imagunit\,\tsc\, \widehat{u}' (\bsh,t)
  =
  2\pi^2\tsc^2 \;\|\bsh\|^2_2 \; \widehat{u} (\bsh,t) + \widehat{f}(\bsh,t),
\end{equation}
for all $\bsh\in\Z^d$, with $\widehat{u}'(\bsh,t) = (\partial/\partial t)\,\widehat{u}(\bsh, t)$ and $\widehat{f}(\bsh,t)$ the Fourier coefficients of $f(\bsx,t) := u(\bsx,t) \, v(\bsx)$.
\end{lemma}

We then truncate the Fourier series~\eqref{eq:u-series} to a finite sum on an anti-aliasing set $\calA(\bsz,n)$ associated to a rank-$1$ lattice $\Lambda(\bsz,n)$ to get the approximation
\begin{align}\label{approx_u}
  u_a(\bsx,t)
  :=
  \sum_{\bsh\in \calA({\bsz,n})} \widehat{u}_a(\bsh,t) \, \exp(\twopii\bsh\cdot\bsx)
  ,
\end{align}
with the approximated coefficients calculated by the rank-$1$ lattice rule
\begin{align}\label{approx_uhat}
  \widehat{u}_a(\bsh,t)
  :=
  {\frac1n}
  \sum_{\bsp\in \Lambda({\bsz,n})}
    u(\bsp,t) \, \exp(-\twopii\bsh\cdot\bsp)
  .
\end{align}
The subscript $a$ of $u_a(\bsx,t)$ and $\widehat{u}_a(\bsh,t)$ indicates that these are approximations of $u(\bsx,t)$ and $\widehat{u}(\bsh,t)$ respectively.
For simplicity of notation, we omit the time $t$ in the rest of this section.
Due to the rank-$1$ lattice structure and by choosing the anti-aliasing set to be of full size, we have the following properties:
\begin{theorem}\label{thm:rank-1-properties}
Given a rank-$1$ lattice point set $\Lambda(\bsz,n)$ and a corresponding anti-aliasing set $\calA(\bsz,n)$ with $|\calA(\bsz,n)| = n$, the following properties hold.
\\[2mm]
(\textit{i}) (Character property and dual character property)
For any two vectors $\bsh,\bsh' \in \calA(\bsz,n)$
\begin{align}
  \frac{1}{n}
  \sum_{\bsp\in\Lambda(\bsz,n)} \exp(\twopii (\bsh-\bsh')\cdot\bsp)
  =
  \delta_{\bsh,\bsh'},
  \end{align}
  where $\delta_{\bsp,\bsp'}$ is the Kronecker delta function that is $1$ if $\bsp=\bsp'$ and $0$ otherwise.
Also, for any two lattice points $\bsp,\bsp' \in \Lambda(\bsz,n)$
\begin{align}
  \frac{1}{n}
  \sum_{\bsh\in\calA(\bsz,n)} \exp(\twopii \bsh\cdot(\bsp-\bsp'))
  =
  \delta_{\bsp,\bsp'}.
\end{align}
\\[2mm]
(\textit{ii}) (Interpolation condition)
If $u_a$ is the approximation of a function $u \in A(\T^d)$ by truncating its Fourier series expansion to the anti-aliasing set $\calA(\bsz,n)$ and by calculating the coefficients by the rank-$r$ lattice rule, cfr.~\eqref{approx_u} and~\eqref{approx_uhat}, then for any $\bsp \in \Lambda(\bsz,n)$
\begin{equation}\label{interpolation}
  u_a(\bsp)
  =
  u(\bsp)
  .
\end{equation}
\\[2mm]
(\textit{iii}) (Mapping through FFT)
Define the following vectors:
\begin{align*}
  \bsu
  &:=
  \big( u(\bsp_{k}) \big)_{k=0,\ldots,n-1}
  ,
  \\
  \bsu_a
  &:=
  \big( u_a(\bsp_{k}) \big)_{k=0,\ldots,n-1}
  ,
  \\
  \widehat{\bsu}_a
  &:=
  \big( \widehat{u}_a(\bsh_{\xi}) \big)_{\xi=0,\ldots,n-1}
  ,
\end{align*}
with $\bsp_k = \bsz k / n \bmod{1} \in \Lambda(\bsz,n)$, and where $\bsh_{\xi} \in \calA(\bsz,n)$ is chosen such that $ \bsh\cdot\bsz_\xi \equiv \xi \pmod{n}$.
Then $\bsu = \bsu_a$ (by (ii)) is the collection of function values $u(\bsp)$ on the lattice points $\bsp \in \Lambda(\bsz,n)$ and $\widehat{\bsu}_a$ is the collection of Fourier coefficients $\widehat{u}_a(\bsh)$ (by using the lattice rule, cfr.~\eqref{approx_u} and~\eqref{approx_uhat}) on the anti-aliasing indices $\bsh \in \calA(\bsz,n)$.
The $1$-dimensional discrete Fourier transform and its inverse now maps $\bsu_a \in \C^{n}$ to $\widehat{\bsu}_a \in \C^{n}$ and back.
\\[2mm]
(\textit{iv}) (Aliasing)
  The approximated Fourier coefficients~\eqref{approx_uhat} through the lattice rule $\Lambda(\bsz,n)$ alias the true Fourier coefficients in the following way
\begin{align*}
    \widehat{u}_a(\bsh)
    =
    \sum_{\bsh' \in \Lambda^{\bot}(\bsz,n)} \widehat{u}(\bsh+\bsh')
    =
    \widehat{u}(\bsh) + \sum_{\bszero \ne \bsh' \in \Lambda^{\bot}(\bsz,n)} \widehat{u}(\bsh+\bsh')
    .
\end{align*}
\end{theorem}
\begin{proof}
We refer to \cite[Theorem~2 and Lemma~3]{2018arXiv180806357S} where more general statement for rank-$r$ lattices can be found.
\end{proof}
We remark that the above theorem can also be understood in terms of Fourier analysis on a finite Abelian group where the group, normally denoted as $G$, is the rank-$1$ lattice point set $\Lambda(\bsz,n)$ and the associated character group (Pontryagin dual) $\widehat{G}:=\{\exp(\twopii \bsh \cdot \circ) : \bsh\in\calA(\bsz,n) \}$ with $|\calA(\bsz,n)|=n$. The (dual) character property is then to be understood as orthonormality of $\widehat{G}$ on $L_2 (G)$. The interpolation condition can be seen as the representability of functions by using Fourier series. Due to this structure, the Plancherel theorem also holds:
\[
\sum_{\bsp\in \Lambda(\bsz,n)}  f(\bsp) \, \overline{g(\bsp)}
=
\sum_{\bsh\in \calA(\bsz,n)}  \widehat{f_a}(\bsh) \, \overline{\widehat{g_a}(\bsh)}
\]
for $f, g \in L_2(G)$.

For readers who are not familiar with Fourier transforms on a rank-$1$ lattice, one intuitive way of seeing why one-dimensional FFTs are available is the following. The usual one-dimensional Fourier transform for equidistant points which is a scalar multiple of a unitary Fourier transform, for a function $f:\T \to \C$, can be written as
\[
\widehat{f}(h)= \frac{1}{n} \sum_{k=0}^{n-1}  f(k/n) \exp(-\twopii h k/n ) ,
\]
and the inverse
\[
f(k/n)= \sum_{h=0}^{n-1}  \widehat{f}(h) \exp(\twopii h k/n ) .
\]
Now we see that the Fourier transform on a rank-$1$ lattice has the exact same structure for a function $f:\T^d \to \C$, 
\[
\widehat{f}(\bsh_\xi)= \frac{1}{n} \sum_{k=0}^{n-1}  f(\bsp_k) \exp(-\twopii \bsh_\xi \cdot \bsp_k )= \frac{1}{n} \sum_{k=0}^{n-1}  f(\bsp_k) \exp(-\twopii \xi k/n ),
\]
and
\[
f(\bsp_k)= \sum_{\xi=0}^{n-1}  \widehat{f}(\bsh_\xi) \exp(\twopii \bsh_\xi \cdot \bsp_k )=  \sum_{k=0}^{n-1}  \widehat{f}(\bsh_\xi) \exp( \twopii \xi k/n ),
\]
where we note $\bsp_k= \bsz k / n \bmod{1}$ and $ \bsh_\xi\cdot\bsz \equiv \xi \pmod{n}$. Hence we only need one-dimensional FFTs to transform functions on $\T^d$.

\subsection{Higher-order exponential splitting}

For the temporal discretization, we employ a higher-order exponential splitting scheme (also called an exponential propagator), see, e.g., \cite{bandrauk1992higher,MR2399406,MR1059400}.
To describe the higher-order exponential splitting, let us consider the following ordinary differential equation:
\begin{align}\label{eq:ODE}
   y'(t) = (A + B) \, y(t)
  ,
  \qquad y(0) = y_0,
\end{align}
where $A$ and $B$ are differential operators. The solution for the equation~\eqref{eq:ODE} is $y(t)=\rme^{(A+B)t}y_0$. However, often it is not possible to compute this exactly, and one needs to approximate the quantity with cheap computational cost.
When both $\rme^{At}$ and $\rme^{Bt}$ can be computed easily, the higher-order exponential splitting is a powerful tool to approximate the solution $\rme^{(A+B)t}y_0$. The approximated solution for this case is given by:
\begin{align}\label{eq:HOsplitting}
y(t+\Deltat)
\approx
\rme^{b_1 B \, \Deltat} \, \rme^{a_1 A \, \Deltat} \, \cdots \rme^{b_s B \, \Deltat} \rme^{a_s A \, \Deltat}
  \, y(t),
\end{align}
where $a_i$ and $b_i$, $i=1,\ldots,s$, are coefficients determined by the desired order of convergence $p$. In other words, if the splitting~\eqref{eq:HOsplitting} satisfies 
\begin{align}\label{eq:local}
 \|
 \rme^{b_1 B \, \Deltat} \, \rme^{a_1 A \, \Deltat} \, \cdots \rme^{b_s B \, \Deltat} \rme^{a_s A \, \Deltat}
  \, y(t)
  -
  \rme^{(A+B)\Deltat} \, y(t)
  \|_{X}
  \le
  C (\Deltat)^{p+1},
\end{align}
for some normed space $X$, where the constant $C$ is independent of $\Deltat$, then the splitting is said to have $p$-th order. The number of steps $s$ and the coefficients $a_i$, $b_i$ can be determined according to the order $p$, see \cite{MR2221614} for details.
We evolve the time using this discretization from time $0$, i.e., 
\[
y_{k+1}
=
\rme^{b_1 B \, \Deltat} \, \rme^{a_1 A \, \Deltat} \, \cdots \rme^{b_s B \, \Deltat} \rme^{a_s A \, \Deltat}
\, y_k
,
\qquad
y_{\{k=0\}} = y_0.
\]
By summing up the local errors~\eqref{eq:local} of each step $k=1,\ldots,m$, where $t=m\Delta t$, gives the total error: 
\[
\|
y_m - y(t)
\|_{X}
\le
C \;m\Deltat\; (\Deltat)^p
=
C t (\Deltat)^p.
\]
We call this quantity the total error in the $L_2$ sense, and this is the reason why the splitting is called to be of $p$-th order.
The error coming from the exponential splitting can be related to commutators of two operators $A$ and $B$, namely $[A,B]:=AB-BA$, $[A,[A,B]]:=A^2B-2ABA+BA^2$, etc. We introduce the notation for the $p$-th commutator by following \cite{MR2399406}: 
\[
 \ad_{A}^p(B)=[A,\ad_{A}^{p-1}(B)],\; \; \ad_{A}^0(B)=B, 
\]
where $p\ge 1$.
When the $p$-th commutator is bounded, it is known that the $p$-th order exponential splitting gives the desired order, see \cite[Lemma~1 and Theorem~1]{MR2399406}. We also refer to \cite[Theorem~2.1]{MR1799313} for the second-order splitting (namely, Strang splitting) in a more abstract setting. 

\subsection{Higher-order exponential splitting on rank-$1$ lattices}

We apply the higher-order exponential splitting to the space discretized TDSE in this section.
For solving the TDSE~\eqref{eq:tdse} in the dual space with finite number of Fourier basis functions, we will rewrite the problem in vector form.
We let $\widehat{\bsu}_t := \bigl( \widehat{u}_a(\bsh_{0},t),$ $\dots,\widehat{u}
_a(\bsh_{n-1},t) \bigr)$ the approximated solution at time $t$. Throughout   time evolution, we use a fixed anti-aliasing set $\calA(\bsz,n) = \{ \bsh_\xi \mathrel{:} \xi = 0, \ldots, n-1\}$ of full size $| \calA(\bsz,n) | = n$, where we denote $\bsh_\xi \in \calA(\bsz,n)$ as such a vector that $\bsh_\xi \cdot \bsz \equiv \xi \pmod{n}$. 
We obtain the following relation by imposing that~\eqref{eq:Fourier-ODE} holds for all $\bsh\in\calA(\bsz,n)$,
\begin{align}\label{ode}
  \imagunit\,\tsc\, \widehat{\bsu}_t'
  &=
  \frac{1}{2}\tsc^2D_n\widehat{\bsu}_t+W_n\widehat{\bsu}_t,
\end{align}
with the initial condition $\widehat{\bsu}_{0} =\widehat{\bsg}_a:=(\widehat{g}_a(\bsh_{0}),\dots,\widehat{g}_a(\bsh_{n-1}))$,
\begin{align}\label{eq:Dn}
  D_n
  :=
  \diag\left((4\pi^2\|\bsh_\xi\|_2^2)_{\xi=0,\ldots,n-1}\right)
  ,
\end{align}
and the potential multiplication operator $W_n:=F_n V_n F_n^{-1}$ with 
\begin{align}\label{eq:Vn}
  V_n
  :=
  \diag\left( \left( v(\bsp_k) \right)_{k=0,\ldots,n-1} \right)
  ,
\end{align}
where $F_n$ is the unitary Fourier matrix
\[
F_n
=
\left(\frac{1}{\sqrt{n}}
       \exp(-\twopii \xi\xi'/n)
     \right)_{\xi,\xi'=0,\ldots,n-1}.
 \]
The approximation of the multiplication operator, $W_n$, is justified by the following lemma which is taken from \cite{2018arXiv180806357S}.
\begin{lemma}[Multiplication operator on rank-$1$ lattices]
Given a rank-$1$ lattice point set $\Lambda(\bsz, n)$ and corresponding anti-aliasing set $\calA(\bsz,n)$ of full size, a potential function $v \in E_\alpha(\T^d)$ with $\alpha \ge 2$ and a function $u_a \in E_\beta(\T^d)$ with $\beta \ge 2$ with Fourier coefficients only supported on $\calA(\bsz,n)$.
Then the action in the Fourier domain restricted to $\calA(\bsz,n)$ of multiplying with $v$, that is $f_a(\bsx) = v(\bsx) \, u_a(\bsx)$, on the nodes of the rank-$1$ lattice, and with $f_a$ having Fourier coefficients restricted to the set $\calA(\bsz,n)$, can be described by a circulant matrix $W_n \in \C^{n\times n}$ with $W_n = F_n \, V_n \, F_n^{-1}$, with $V_n$ given by~\eqref{eq:Vn} and $F_n$ the unitary Fourier matrix, where the element at position $(\xi,\xi')$ of $W_n$ is given by
\begin{align}\label{eq:Wn-elements}
  w_{\xi,\xi'}
  &
  =
  w_{(\xi-\xi') \bmod{n}}
  =
  \sum_{\substack{\bsh\in\Z^d \\ \bsh\cdot \bsz\equiv \xi-\xi'\tpmod{n}}}\kern-2em \widehat{v}(\bsh)
  .
\end{align}
\end{lemma}%
\begin{proof}
We refer to \cite[Lemma~5]{2018arXiv180806357S}.
\end{proof}

We approximate the solution of the ordinary differential equation~\eqref{ode}
\[
  \widehat{\bsu}_{t}
  =
  \rme^{-\frac{\imagunit}{\tsc} W_n \, t - \frac{\imagunit\tsc}{2} D_n \, t}
  \, \widehat{\bsu}_{0},
\]
by applying the higher-order exponential splitting method~\eqref{eq:HOsplitting}:
\begin{equation}
  \widehat{\bsu}_a^{k+1}
  =
  \rme^{- b_1 \frac{\imagunit}{\tsc} W_n \, \Deltat}
  \, 
  \rme^{- a_1 \frac{\imagunit\tsc}{2} D_n \, \Deltat}
  \, \cdots \rme^{-b_s \frac{\imagunit}{\tsc} W_n \, \Deltat} \rme^{-a_s \frac{\imagunit\tsc}{2} D_n \, \Deltat}
 \widehat{\bsu}_a^k
  \qquad \text{for } k=0,1,\dots,m-1,
\label{eq:HigherTDSE}
\end{equation}
where again the coefficients $a_i, b_i$ are determined according to the desired order of convergence, and 
\[
  \rme^{-\frac{\imagunit}{2} W_n \Deltat} 
  = 
  F_n \diag\left( (\rme^{-\frac{\imagunit}{2} v(\bsp_k) \Deltat})_{k=0,\ldots,n-1} \right) F_n^{-1}.
\]
The approximated solution at the time $t=k\Deltat$ is then obtained by stepping time $\Deltat$ iteratively by~\eqref{eq:HigherTDSE}.
In the following we show the \emph{commutator bounds} which correspond to \cite[Hypothesis~3]{MR2399406} and lead us to the total bound as in \cite[Theorem~1]{MR2399406}. 

\begin{theorem}[$p$-th commutator bound and total error bound]\label{thm:pth-commu}
Given a rank-$1$ lattice with generating vector $\bsz \in \Z^d$ and modulus $n$ and a TDSE with a potential function $v \in E_\alpha(\T^d)$ with $ \alpha > 2p + 1/2$ and an initial condition $g \in E_\beta(\T^d)$ with $\beta \ge 2$.
Let $D = \tfrac{\tsc}2 D_n$ and $W = \frac1\tsc W_n$ with $D_n$ and $W_n = F_n V_n F_n^{-1}$ as defined in~\eqref{eq:Dn} and~\eqref{eq:Wn-elements}, and with $V_n$ as defined in~\eqref{eq:Vn} using the potential function $v$.

If the anti-aliasing set $\calA(\bsz,n) = \{ \bsh_\xi \in \Z^d : \bsh_\xi \cdot \bsz \equiv \xi \pmod{n} \text{ for } \xi = 0,\ldots,n-1 \}$, with full cardinality, is chosen such that it has minimal $\ell_2$ norm, i.e.,
\begin{align}\label{eq:hrepresenter}
  \| \bsh_\xi \|_2
  =
  \min_{\bsh' \in A(\bsz,n,\xi)} \|\bsh'\|_2
  ,
\end{align}
with
\begin{align*}
  A(\bsz,n,\xi)
  :=
  \bigl\{ \bsh \in \Z^d : \bsh \cdot \bsz \equiv \xi \pmod{n} \bigr\}
  ,
\end{align*}
then for all $\bsy \in \R^n$ we have the following bound for the $p$-th commutator:
\begin{align*}
 \| \ad_{D}^p(W) \;\bsy \|_2 \le c \, \|(D+I)^p\,\bsy\|_2,
 \end{align*}
 where $c$ is a constant independent of $n$ and $\bsy$. \\
 This commutator condition and \cite[Theorem~1]{MR2399406} directly give us the total error bound for (\ref{ode}):
 \[
\|\widehat{\bsu}_{t} - \widehat{\bsu}_{a}^{m} \|_2 
\le
C\|\widehat{\bsu}_{0}-\widehat{\bsu}_{a}^{0} \|_2
+
C' (\Deltat)^p \|(D+I)^p \widehat{\bsu}_{0}\|_2, 
 \]
 where $m\Deltat=t$ and the constants depend on $t$ but not on $m$ or $\Deltat$.
\end{theorem}
\begin{proof}
Let $M:=\ad_{D}^p(W) \; (D+I)^{-p}$. Since $(D+I)^p$ is is non-singular, the claim of the theorem is equivalent to the assertion that
the induced $\ell_2$ norm of the matrix $\| M \|_2:= \sup_{\bszero \ne \bsy \in \R^n} \| M \, \bsy \|_2 / \| \bsy \|_2$ is bounded independent of $n$.
Each element of the matrix $M$ is given by,
\[
 M
 =
 \left(
       \frac{
         (\|\bsh_\xi\|^2_2-\|\bsh_{\xi'}\|^2_2)^p
       }{
         \tsc (\|\bsh_{\xi'}\|^2_2 + c_1)^p
       }
       \, w_{\xi,\xi'}
     \right)_{\xi,\xi'=0,\ldots,n-1}
  ,
\]
where the constant $c_1= 1/(2\pi\tsc)^p > 0$. Now we bound $\| M \|_2$ by using $\| M \|_2 \le \sqrt{\|M\|_1 \|M\|_{\infty}}$.
First we bound $\| M \|_1$: 
\begin{align*}
\|M\|_1
&=
\frac{1}{\tsc}\max_{\xi'\in\Z_n} \sum_{\substack{\xi=0\\\xi\ne\xi'}}^{n-1} \left| \frac{( \|\bsh_\xi\|^2_2-\|\bsh_{\xi'}\|^2_2 )^p}{(\|\bsh_{ \xi'}\|^2_2 + c_1 )^p} \, w_{\xi,\xi'} \right| \\
&\le 
\frac{1}{\tsc}\max_{\xi'\in\Z_n} \sum_{\substack{\xi=0\\\xi\ne\xi'}}^{n-1} \left| \frac{( \max ( \|\bsh_\xi\|^{2p}_2, \|\bsh_{\xi'}\|^{2p}_2)}{(\|\bsh_{ \xi'}\|^2_2 + c_1 )^p} \, w_{\xi,\xi'} \right|.
\end{align*}
We notice that the diagonal components of $M$ ($\xi=\xi'$) is always $0$, hence we exclude such cases in the following argument.
Because we collect the anti-aliasing set by minimizing the $\ell_2$ norm~\eqref{eq:hrepresenter}, we have
$\|\bsh_\xi\|_2 \le \|\bsh_\xi'\|_2$ for any $\bsh_\xi' \in A(\bsz,n,\xi)$.
In particular, this holds for $\bsh_\xi' = \bsh_{\xi-\xi'} + \bsh_{\xi'}$ since $(\bsh_{\xi-\xi'}+\bsh_{\xi'}) \cdot \bsz \equiv \xi \pmod{n}$ for any choice of $\xi'=0,\ldots,n-1$. 
This gives us the connection between $\| \bsh_\xi \|_2$ and $\| \bsh_{\xi'} \|_2$ using $\| \bsh_{\xi-\xi'} \|_2$:
\begin{align*}
\frac{\|\bsh_\xi\|^2_2}{\|\bsh_{ \xi'}\|^2_2 + c_1}
\le
\frac{\|\bsh_{\xi'}+\bsh_{\xi-\xi'}\|^2_2}{\|\bsh_{ \xi'} \|^2_2 + c_1}
\le
4\| \bsh_{\xi-\xi'}\|^2_2,
\end{align*}
for $\xi\ne\xi'$. We continue from the above bound of $\| M \|_1$,
\begin{align*}
\|M\|_1
&\le 
\frac{1}{\tsc}\max_{\xi'\in\Z_n} \sum_{\substack{\xi=0\\\xi\ne\xi'}}^{n-1} \left| \frac{( \max ( \|\bsh_\xi\|^{2p}_2, \|\bsh_{\xi'}\|^{2p}_2)}{(\|\bsh_{ \xi'}\|^2_2 + c_1 )^p} \, w_{\xi,\xi'} \right| \\
&\le
\frac{1}{\tsc}\max_{\xi'\in\Z_n} \sum_{\substack{\xi=0\\\xi\ne\xi'}}^{n-1} \left| \max \left(\frac{\|\bsh_\xi\|^2_2}{\|\bsh_{ \xi'}\|^2_2 +c_1 },1 \right)^p \, w_{\xi,\xi'} \right|
\\
&\le
\frac{1}{\tsc}\max_{\xi'\in\Z_n} \sum_{\substack{\xi=0\\\xi\ne\xi'}}^{n-1} \left| \max \left( 4^p \| \bsh_{\xi-\xi'}\|^{2p}_2,1 \right) \, w_{\xi,\xi'} \right|
\\
&=
\frac{1}{\tsc}\max_{\xi'\in\Z_n} \sum_{\substack{\xi=0\\\xi\ne\xi'}}^{n-1} \left| \left( 4^p \| \bsh_{\xi-\xi'}\|^{2p}_2 \right) \, w_{\xi,\xi'} \right|
\\
&\le
\frac{4^p}{\tsc}\max_{\xi'\in\Z_n} \sum_{\substack{\xi=0\\\xi\ne\xi'}}^{n-1}  \| \bsh_{\xi-\xi'}\|^{2p}_2  \,  \Biggl|
     \sum_{\bsh\in A(\bsz,n,\xi-\xi')} \widehat{v}(\bsh) \Biggr|
\\
&\le
\frac{4^p}{\tsc}\max_{\xi'\in\Z_n} \sum_{\substack{\xi=0\\\xi\ne\xi'}}^{n-1} \sum_{\bsh\in A(\bsz,n,\xi-\xi')}  \| \bsh \|^{2p}_2  \,  | \widehat{v}(\bsh) |
\\
&\le
\frac{4^p}{\tsc}\sum_{\bsh\in \Z^d} \|\bsh\|^{2p}_2 \, | \widehat{v}(\bsh)|.
\end{align*}
For the last inequality, we used the conjugacy decomposition~\eqref{eq:disjunct}. By using Cauchy--Schwarz inequality and multiplying and dividing by $r_\alpha$, we have
\begin{align*}
 \sum_{\bsh\in \Z^d} \|\bsh\|^{2p}_2 \, | \widehat{v}(\bsh)|
 &\le
   \left( 
    \sum_{\bsh\in \Z^d} r^2_\alpha(\bsh) \, |\widehat{v}(\bsh)|^2
  \right)^{1/2}
  \left(
    \sum_{\bsh\in \Z^d} \frac{\|\bsh\|^{4p}_2}{r^2_\alpha(\bsh)}
  \right)^{1/2}
  \\
&\le
  \|v\|_{E_\alpha(\T^d)}
  \left(
     \sum_{\bsh\in \Z^d} \frac{(\sqrt{d} \, \|\bsh\|_\infty)^{4p}}{r^2_\alpha(\bsh)} 
  \right)^{1/2}
  \\
&\le
  \|v\|^2_{E_\alpha(\T^d)}
  \left(
    \sum_{\bsh\in \Z^d} \frac{d^{2p}}{r^2_{\alpha-2p}(\bsh)}  \right )^{1/2} \\
&\le
  \|v\|^2_{E_\alpha(\T^d)}
  \left( d^{2p} \, (1 + 2 \, \zeta(2\alpha-4p))^d \right)^{1/2} 
  <
  \infty
  .
\end{align*}
This means we have bounded $\|M\|_1$ independent of $n$.
For $\|M\|_\infty$ we can proceed in a similar way to obtain
\begin{align*}
  \|M\|_\infty
  &=
  \max_{\xi\in\Z_n} \sum_{\substack{\xi'=0\\\xi'\ne\xi}}^{n-1}
  \left|
  \frac{ \left(\|\bsh_{\xi}\|^2_2 - \|\bsh_{\xi'}\|^2_2\right)^p}
       {\left(\|\bsh_{\xi'}\|^{2}_2 + c_1\right)^p} \, w_{\xi,\xi'}
  \right|
 \\
  &\le
  \frac{4^p}{\tsc} \|v\|^2_{E_\alpha(\T^d)}
  \left( d^{2p} \, (1 + 2 \, \zeta(2\alpha-4p))^d \right)^{1/2} <\infty
  .
\end{align*}
Therefore, we have $\|M\|_2 < \infty$ independent of $n$.
The total error bound directly follows from this commutator bound and \cite[Theorem~1]{MR2399406}.
\end{proof}

\section{Numerical results} \label{sec:numerical}

We demonstrate our method by showing some numerical results in this section. We construct rank-$1$ lattices by using the component-by-component (CBC) construction
\cite{MR2272256,Nuy2014}. The code for producing the rank-$1$ lattice is available online \cite{MR2198499}, \texttt{fastrank1expt.m}. With the script, we choose $n$ being a power of $2$ and generate the vector $\bsz$ which is optimized for integration in (unweighted) Korobov space with first order mixed derivatives, i.e., $\alpha = 1$.
In Table~\ref{tb:parameter} we display the generating vector $\bsz$ and the number of points $n$ for the following numerical results.
Using given $n$ and $\bsz$, we construct the anti-aliasing set in accordance with Theorem~\ref{thm:pth-commu} in the following manner: (i) first we generate all integer vector $\bsh\in\Z^d$ in a bounded region $\|\bsh\|\le R$ for a well chosen $R$; (ii) then we sort the obtained set according to the $\ell_2$ distance in ascending order; (iii) we calculate the value $m_{\bsh}:=\bsh\cdot\bsz \bmod{n}$ in the sorted order and store $\bsh$ in $\calA(\bsz,n)$ if the value $m_{\bsh}$ has not appeared before. We repeat this step (iii) until we have the full cardinality $|\calA(\bsz,n)|=n$.

\begin{table}[]
\centering
\label{tb:parameter}
\begin{tabular}{c|c|c}\hline
 $d$ & $n$  & $\bsz^{\top}$ \\\hline
$2$  & $2^{16} $& $(1,100135)$   \\[2mm]
$4$ & $2^{20}$ & $ (1, 443165,95693,34519)$ \\[2mm]
 $6$  &$2^{24}$ & $(1, 6422017, 7370323, 2765761, 8055041, 2959639)$   \\[2mm]
 $8$ & $2^{24}$ & $(1, 6422017, 7370323, 2765761, 8055041, 2959639, 7161203, 4074015)$ \\[2mm] 
\hline
\end{tabular}
\caption{Parameters of the rank-$1$ lattice points for our numerical results. }
\end{table}

\subsection{Convergence with respect to time step size}

We consider a common numerical setting as it is considered in \cite{MR2377257,MR1799313,MR2399406} where Fourier pseudo-spectral methods are used. We calculate the error with different value of time steps against a reference solution. 
For the initial condition $g(\bsx)$, we choose the \emph{Gaussian wave packet} given by:
\[
  g(\bsx)
  :=
  \left(\frac{2}{\pi \tsc}\right)^{d/4}\exp\left(-\frac{ \sum_{j=1}^{d}\left(2 \pi x_j-\pi\right)^2}{\tsc}\right)\frac{1}{c},
\]
where the constant $c$ is a normalizing constant to make $\|g \|_{L_2}=1$.
For the potential function $v$, we consider a \emph{smooth potential} function
\[
  v_1(\bsx)
  =
  \prod_{j=1}^d (1-\cos(2\pi x_j))
  ,
\]
and a \emph{harmonic potential} function
\[
  v_2(\bsx)
  =
  \frac{1}{2}\sum^d_{j=1}(2\pi x_j-\pi)^2
  .
\]
Our aim is to show the temporal discretization error $\|u_a(\bsx,t)-u_a^m (\bsx)\|_{L_2}$ at fixed time $t=m\,\Deltat=1$, for that sake we calculate a reference solution $u^{M}_a(\bsx)$ with the finest time step size $\Deltat=1/M=1/10000$, as an approximation of $u_a(\bsx,t)$.
We then vary the time step size $\Deltat=1/m=1/5,\ldots,1/1000$ and calculate $u^{m}_a(\bsx)$ to see the convergence plot of $\|u^M_a(\bsx)-u_a^{m}(\bsx)\|_{L_2}$.

\subsection{Sixth-order splitting}

We recall that the higher-order exponential splitting is written as
\[
y_{k+1}
=
\rme^{b_1 B \, \Deltat} \, \rme^{a_1 A \, \Deltat} \, \cdots \rme^{b_s B \, \Deltat} \rme^{a_s A \, \Deltat}
\, y_k
.
\]
For the sixth-order method, we employ the coefficients $a_j$ and $b_j$ from \cite{MR1423077} denoted as ``s9odr6a'' therein. We exhibit the coefficients in Table~\ref{tb:sixth}.
We plot the results for dimension $2$ to $8$ in Fig~\ref{fig:sixth}. The potential $v_1$ is not smooth enough on the boundary of $[0,1]^d$ so it does not satisfy the required condition in the strict sense. The initial condition $g$ and the potential $v_2$ meet all the required conditions.
The expected sixth-order convergence is consistent in every plot. When the error reaches to the machine precision, the plot becomes flat. For the $2$-dimensional case with the potential $v_2$, we see the convergence happening when the time step size is very small. This can be explained by a phenomenon, called instability of exponential splitting;
this is caused by negative coefficients of the exponential splitting $a_j$ and $b_j$, and is discussed in e.g., \cite{MR3695212}. Especially in \cite{MR3695212}, commutator-free quasi-Magnus exponential integrators are proposed to avoid the issue, however, this is out of the scope of the present paper. The instability issue does not happen in a higher-dimensional settings.

\begin{table}[bpt]
\centering
\begin{tabular}{|c|c|c|c|} \hline
 & $a_j$ & & $b_j$\\ \hline
$j=1,9$ & 0.392161444007314 & $j=1,10$ & 0.196080722003657 \\
$j=2,8$ & 0.332599136789359 & $j=2,9$ & 0.362380290398337  \\
$j=3,7$ & -0.706246172557639 & $j=3,8$ & -0.186823517884140 \\
$j=4,6$ & 0.0822135962935508 & $j=4,7$ & -0.312016288132044 \\
$j=5$  & 0.798543990934830 & $j=5,6$ & 0.440378793614190 \\
$j=10$ & 0 &  & \\ \hline 
\end{tabular}
\caption{ Coefficients for the sixth-order method, calculated based on \cite{MR1423077}. }\label{tb:sixth}
\end{table}

\begin{figure}
\centering
\ifpreprint
  \includegraphics{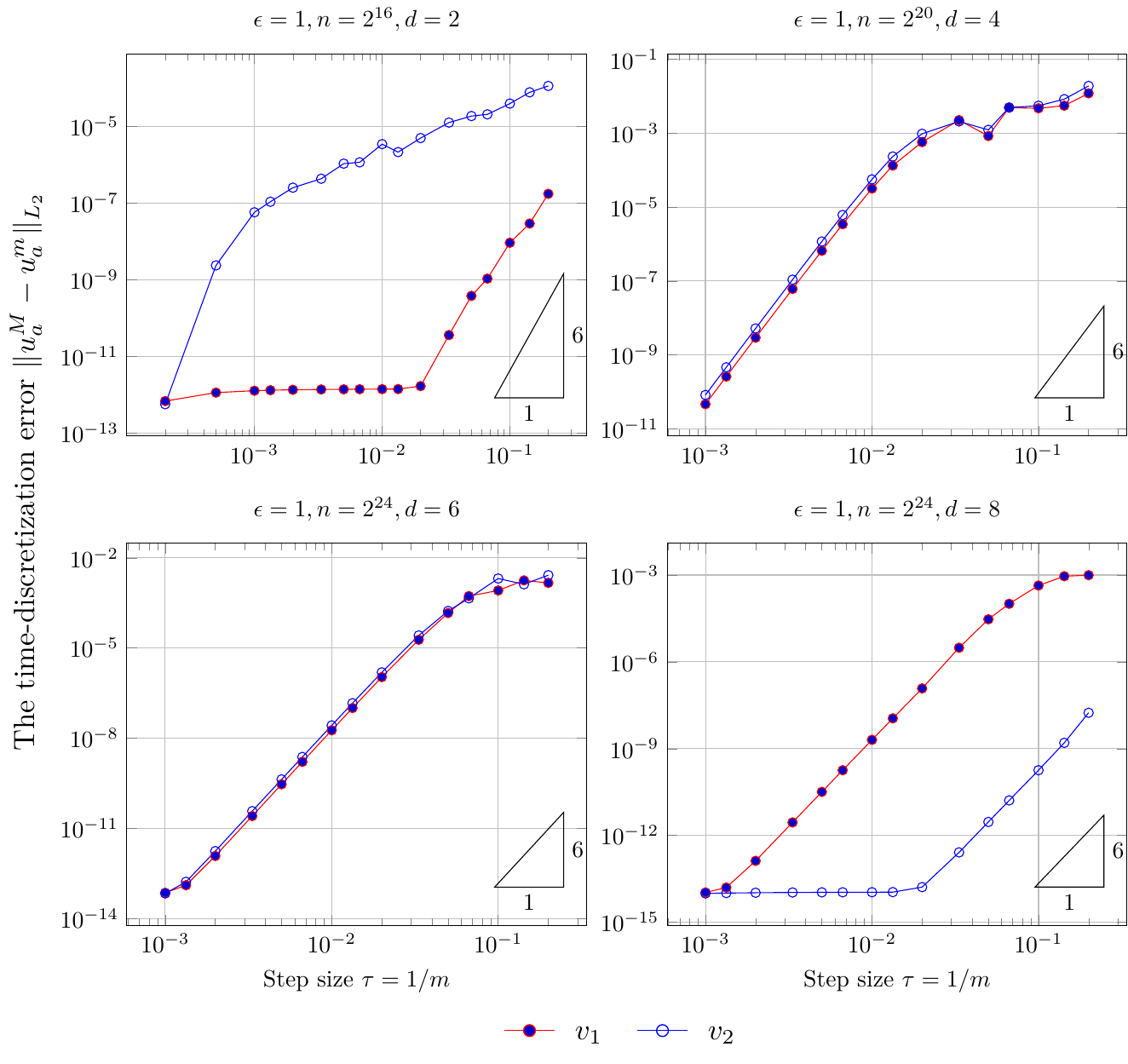}
\else
\begin{tikzpicture}[scale=0.8]
\begin{groupplot}[ group style={group name=myplots, group size=2 by 2,vertical sep=1.6cm,horizontal sep=1.2cm },legend style={font=\fontsize{4}{5}\selectfont}]
\nextgroupplot[title={$\epsilon = 1, n=2^{16},d=2$}, xmajorgrids,
ymajorgrids,       xmode=log,
        ymode=log,ylabel={},xlabel={}]
\pgfplotstableread{6thgs2_d2 epsilon1N65536.dat}{\datatable}
\addplot+[color=red, mark=*] table  {\datatable};
\pgfplotstableread{6thghat2_d2 epsilon1N65536.dat}{\datatable}
\addplot+[color=blue, mark=o] table  {\datatable};
\logLogSlopeTriangle{0.95}{0.15}{0.1}{6}{black};

\nextgroupplot[title={$\epsilon = 1, n=2^{20},d=4$}, xmajorgrids,
ymajorgrids,       xmode=log,
        ymode=log,ylabel={},xlabel={}]
\pgfplotstableread{6thgs2_d4 epsilon1N1048576.dat}{\datatable}
\addplot+[color=red, mark=*] table  {\datatable};
\pgfplotstableread{6thghat2_d4 epsilon1N1048576.dat}{\datatable}
\addplot+[color=blue, mark=o] table  {\datatable};
\logLogSlopeTriangle{0.95}{0.15}{0.1}{6}{black};

\nextgroupplot[title={$\epsilon = 1, n=2^{24},d=6$}, xmajorgrids,
ymajorgrids,       xmode=log,
        ymode=log,ylabel={},xlabel={Step size $\tau=1/m$}]
\pgfplotstableread{6thgs2_d6 epsilon1N16777216.dat}{\datatable}
\addplot+[color=red, mark=*] table  {\datatable};\label{gs1}
\pgfplotstableread{6thghat2_d6 epsilon1N16777216.dat}{\datatable}
\addplot+[color=blue, mark=o] table  {\datatable};\label{gh1}
\logLogSlopeTriangle{0.95}{0.15}{0.1}{6}{black};

\nextgroupplot[ title={$\epsilon = 1, n=2^{24},d=8$},xmajorgrids,
ymajorgrids,       xmode=log,
        ymode=log,xlabel={Step size $\tau=1/m$}]

\pgfplotstableread{6thgs2_d8 epsilon1N16777216.dat}{\datatable}
\addplot+[color=red, mark=*] table  {\datatable};
\pgfplotstableread{6thghat2_d8 epsilon1N16777216.dat}{\datatable}
\addplot+[color=blue, mark=o] table  {\datatable};
\logLogSlopeTriangle{0.95}{0.15}{0.1}{6}{black};

          \end{groupplot}
    \path (myplots c1r1.outer south west)
          -- node[anchor=south,rotate=90] {{The time-discretization error $\|{u_a^M-u_a^m}\|_{L_2}$}}
          (myplots c1r1.outer south west)
    ;
\path (myplots c1r1.south west|-current bounding box.south)--
      coordinate(legendpos)
      (myplots c2r1.south east|-current bounding box.south);
\matrix[
    matrix of nodes,ampersand replacement=\&,
    anchor=north,
    draw=none,
    outer sep=0.5em,
    draw=none
  ]at([below,yshift=1ex]legendpos)
  {
    \ref{gs1}\& $v_1$ \&[5pt]
    \ref{gh1}\& $v_2$ \&[5pt]\\};
\end{tikzpicture}
\fi
\caption{The time-discretization error with the sixth-order method.}
\label{fig:sixth}
 \end{figure}
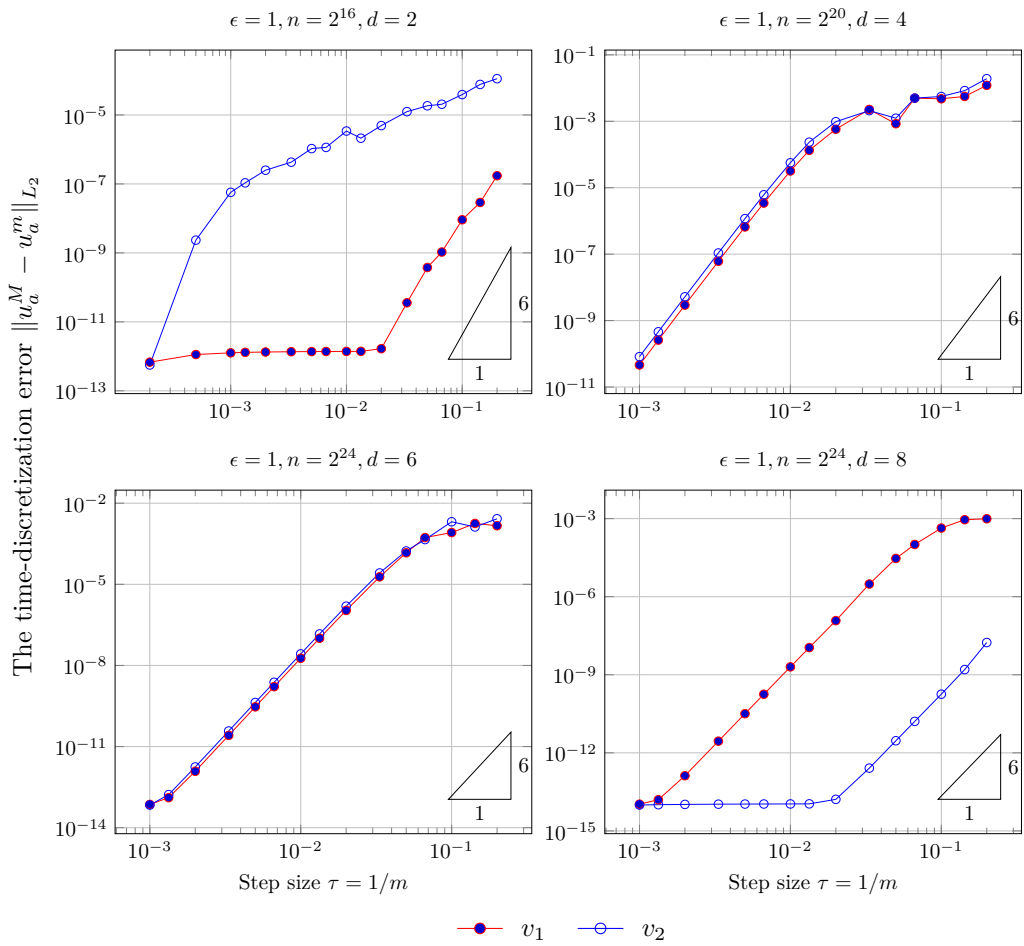
 
\subsection{Eighth-order splitting}

For the eighth-order method, we employ the coefficients again from \cite{MR1423077} denoted as ``s17odr8a''. The coefficients are shown in Table~\ref{tb:eighth}.
The results are shown in Fig~\ref{fig:eighth} and we again see that the convergence rate is consistently eighth order in each plot. Most of the plot seems to be similar to Fig~\ref{fig:sixth} but with faster convergence, therefore they reach to the machine precision more quickly.

\begin{table}[bpt]
\centering
\begin{tabular}{|c|c|c|c|} \hline
 & $a_j$ & & $b_j$\\ \hline
$j=1,17$ & 0.130202483088890 & $j=1,18$ & 0.0651012415444450\\
$j=2,16$ & 0.561162981775108  & $j=2,17$ & 0.345682732431999\\
$j=3,15$ & -0.389474962644847 & $j=3,16$ & 0.0858440095651306\\
$j=4,14$ & 0.158841906555156 & $j=4,15$ & -0.115316528044846\\
$j=5,13$ & -0.395903894133238  & $j=5,14$  & -0.118530993789041\\
$j=6,12$ & 0.184539640978316  & $j=6,13$  & -0.105682126577461\\
$j=7,11$ & 0.258374387686322  & $j=7,12$  & 0.221457014332319\\
$j=8,10$ & 0.295011723609310  & $j=8,11$  & 0.276693055647816\\
$j=9$ & -0.605508533830035  & $j=9,10$  & -0.155248405110362\\
$j=18$ & 0 &  &  \\ \hline 
\end{tabular}
\caption{ Coefficients for the eighth-order method, calculated based on \cite{MR1423077}. }\label{tb:eighth}
\end{table} 
 
 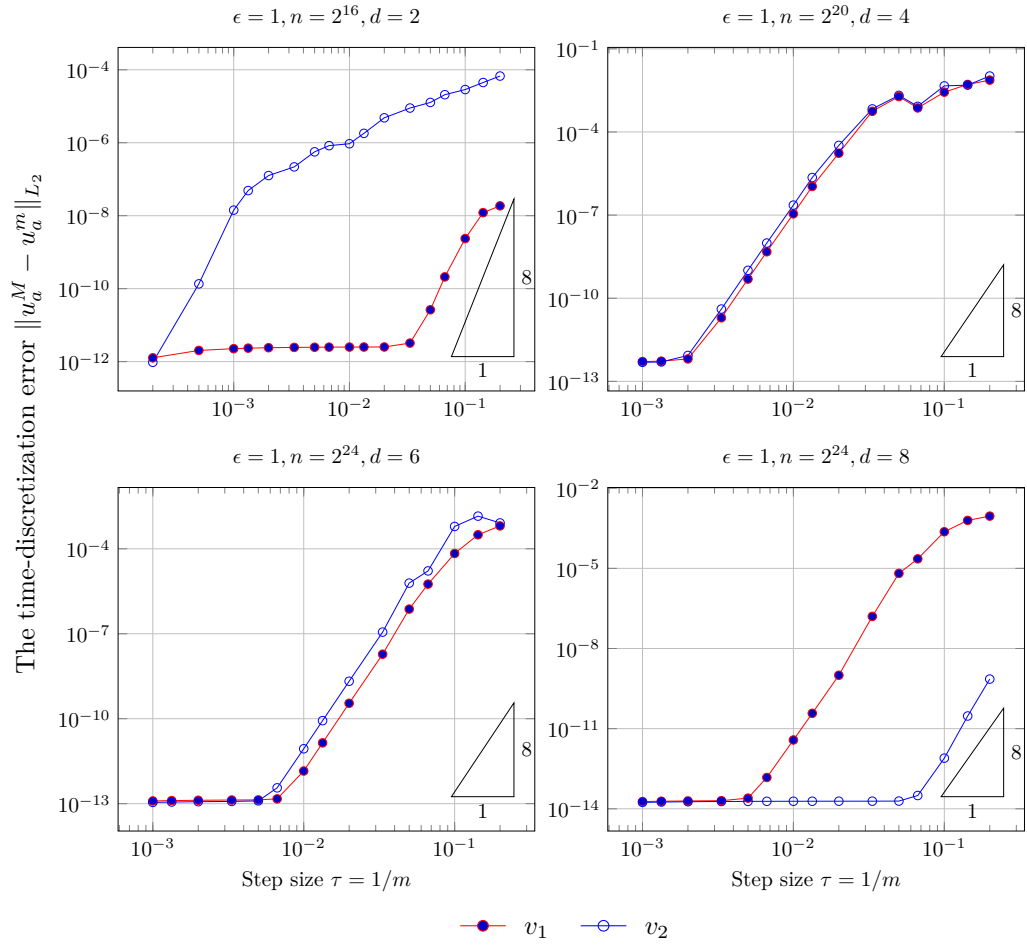
\begin{figure}
\centering
\ifpreprint
  \includegraphics{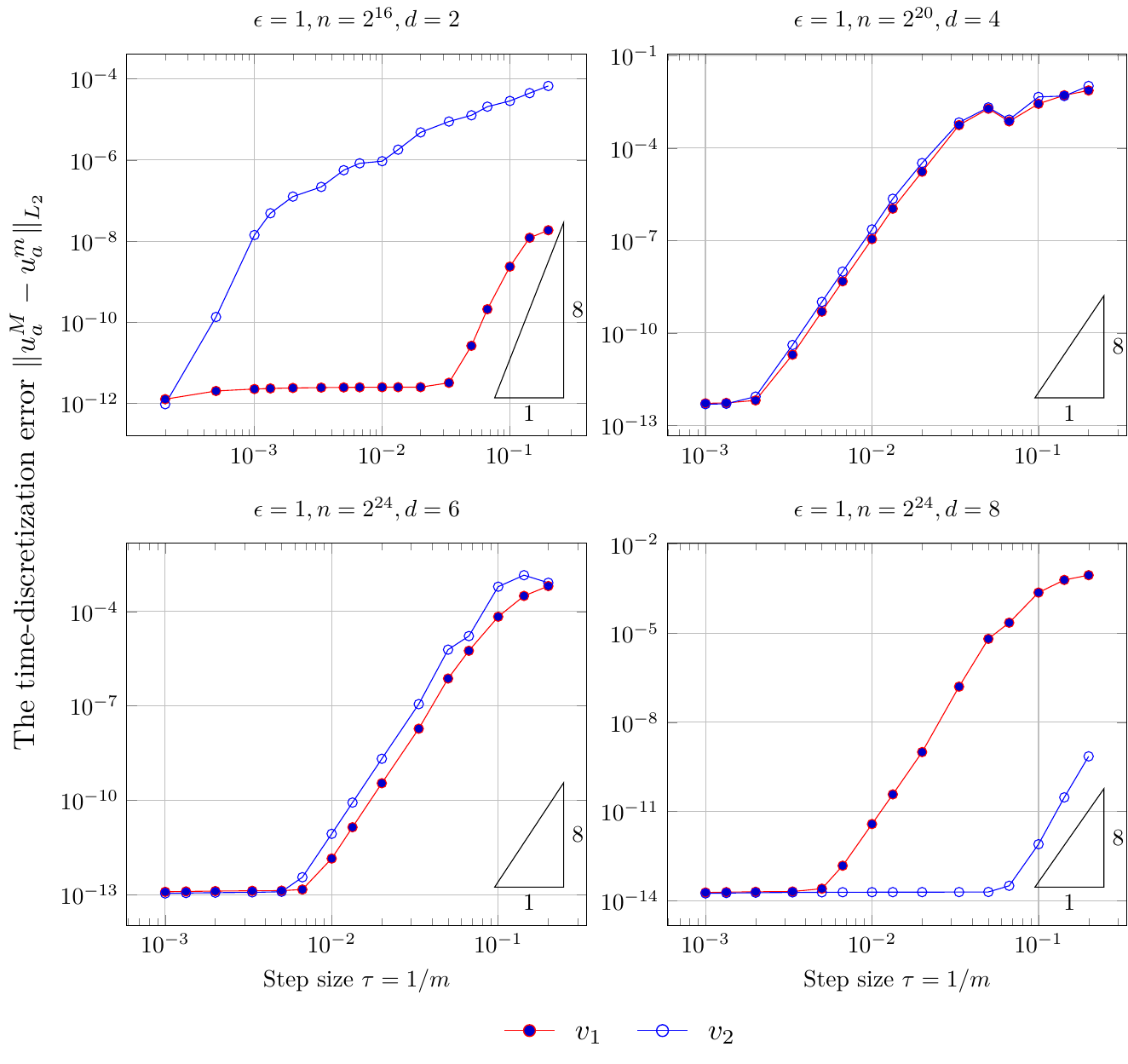}
\else
\begin{tikzpicture}[scale=0.8]
\begin{groupplot}[ group style={group name=myplots, group size=2 by 2,,vertical sep=1.6cm,horizontal sep=1.2cm},legend style={font=\fontsize{4}{5}\selectfont}]

\nextgroupplot[title={$\epsilon = 1, n=2^{16},d=2$}, xmajorgrids,
ymajorgrids,       xmode=log,
        ymode=log,ylabel={}, xlabel={ }]
\pgfplotstableread{8thgs2_d2 epsilon1N65536.dat}{\datatable}
\addplot+[color=red, mark=*] table  {\datatable};
\pgfplotstableread{8thghat2_d2 epsilon1N65536.dat}{\datatable}
\addplot+[color=blue, mark=o] table  {\datatable};
\logLogSlopeTriangle{0.95}{0.15}{0.1}{8}{black};

\nextgroupplot[title={$\epsilon = 1, n=2^{20},d=4$}, xmajorgrids,
ymajorgrids,       xmode=log,
        ymode=log,ylabel={},xlabel={ }]
\pgfplotstableread{8thgs2_d4 epsilon1N1048576.dat}{\datatable}
\addplot+[color=red, mark=*] table  {\datatable};
\pgfplotstableread{8thghat2_d4 epsilon1N1048576.dat}{\datatable}
\addplot+[color=blue, mark=o] table  {\datatable};
\logLogSlopeTriangle{0.95}{0.15}{0.1}{8}{black};

\nextgroupplot[title={$\epsilon = 1, n=2^{24},d=6$}, xmajorgrids,
ymajorgrids,       xmode=log,
        ymode=log,ylabel={},xlabel={Step size $\tau=1/m$}]
\pgfplotstableread{8thgs2_d6 epsilon1N16777216.dat}{\datatable}
\addplot+[color=red, mark=*] table  {\datatable};\label{gs}
\pgfplotstableread{8thghat2_d6 epsilon1N16777216.dat}{\datatable}
\addplot+[color=blue, mark=o] table  {\datatable};\label{gh}
\logLogSlopeTriangle{0.95}{0.15}{0.1}{8}{black};

\nextgroupplot[ title={$\epsilon = 1, n=2^{24},d=8$},xmajorgrids,
ymajorgrids,       xmode=log,
        ymode=log,xlabel={Step size $\tau=1/m$}]

\pgfplotstableread{8thgs2_d8 epsilon1N16777216.dat}{\datatable}
\addplot+[color=red, mark=*] table  {\datatable};
\pgfplotstableread{8thghat2_d8 epsilon1N16777216.dat}{\datatable}
\addplot+[color=blue, mark=o] table  {\datatable};
\logLogSlopeTriangle{0.95}{0.15}{0.1}{8}{black};

          \end{groupplot}
    \path (myplots c1r1.outer south west)
          -- node[anchor=south,rotate=90]{The time-discretization error $\|{u_a^M-u_a^m}\|_{L_2}$}
          (myplots c1r1.outer south west)
    ;
\path (myplots c1r1.south west|-current bounding box.south)--
      coordinate(legendpos)
      (myplots c2r1.south east|-current bounding box.south);
\matrix[
    matrix of nodes,ampersand replacement=\&,
    anchor=north,
    draw=none,
    outer sep=0.5em,
    draw=none
  ]at([below,yshift=1ex]legendpos)
  {
    \ref{gs}\& $v_1$ \&[5pt]
    \ref{gh}\& $v_2$ \&[5pt]\\};
\end{tikzpicture}
\fi
\caption{The time-discretization error with the eighth-order method.}
\label{fig:eighth}
\end{figure}

\section{Conclusion}\label{sec:conclusion}

We proposed a numerical method to solve the TDSE. With our method using the time step size $\Deltat$,
 the temporal discretization error converges like $\mathcal{O}((\Deltat)^p)$ given that the potential function is in Korobov space of smoothness greater than $2p+1/2$.
The numerical results (which are performed from $2$ up to $8$ dimensions) confirmed the theory and the rate of error convergence is consistent.
By using rank-$1$ lattices, calculations of the time stepping operator and multiplications are efficiently done by only using one-dimensional FFTs. 
 
 Pseudo-spectral methods are widely used technique for solving partial differential equations. It is a common choice to use regular grids, but the number of nodes increases exponential with~$d$. 
 We have shown an alternative, rank-$1$ lattice pseudo-spectral methods where the number of points can be chosen freely by the user. In combination with higher-order splitting methods, the proposed method solves the TDSE with higher-order convergence in time.

\end{document}